\documentclass[12pt]{amsart}

\usepackage{amsthm,amssymb}

\theoremstyle{plain}
\newtheorem*{ThmHTtensor}{Theorem 2.3.2}
\newtheorem*{ThmHTschur}{Theorem 2.4.2}
\newtheorem*{Thmssttensor}{Theorem 3.2.1}
\newtheorem*{Thmsstschur}{Theorem 3.3.2}
\newtheorem{Thm}[subsubsection]{Theorem}
\newtheorem{Prop}[subsubsection]{Proposition}
\newtheorem{Lem}[subsubsection]{Lemma}
\newtheorem{Cor}[subsubsection]{Corollary}

\theoremstyle{definition}

\theoremstyle{remark}

\numberwithin{equation}{subsection}

\newcommand{\ol}{\overline}

\newcommand{\mbf}{\mathbf}

\newcommand{\mcl}{\mathcal}
\newcommand{\mfk}{\mathfrak}

\newcommand{\liso}{\overset\sim\gets}
\newcommand{\riso}{\overset\sim\to}

\newcommand{\Z}{\mathbf{Z}}

\newcommand{\Zp}{\mathbf{Z}_{p}}
\newcommand{\Qp}{\mathbf{Q}_{p}}
\newcommand{\Qpbar}{\overline{\mathbf{Q}}_{p}}
\newcommand{\Cp}{\mathbf{C}_{p}}
\newcommand{\CpE}[1]{\mathbf{C}_{p,#1}}

\newcommand{\B}{\mathbf{B}}

\newcommand{\Bcris}{\mathbf{B}_{\cris}}
\newcommand{\BcrisE}[1]{\mathbf{B}_{\cris,#1}}
\newcommand{\Be}{\mathbf{B}_{\e}}
\newcommand{\BeE}[1]{\mathbf{B}_{\e,#1}}

\newcommand{\Bst}{\mathbf{B}_{\st}}
\newcommand{\BstE}[1]{\mathbf{B}_{\st,#1}}

\newcommand{\BdRp}{\mathbf{B}_{\dR}^+}
\newcommand{\BdRpE}[1]{\mathbf{B}_{\dR,#1}^+}
\newcommand{\BdR}{\mathbf{B}_{\dR}}
\newcommand{\BdRE}[1]{\mathbf{B}_{\dR,#1}}

\DeclareMathOperator{\cris}{cris}
\DeclareMathOperator{\e}{e}
\DeclareMathOperator{\st}{st}
\DeclareMathOperator{\dR}{dR}

\DeclareMathOperator{\sen}{sen}

\DeclareMathOperator{\poids}{Wt}
\DeclareMathOperator{\drpoids}{\poids_{\dR}}

\DeclareMathOperator{\Schur}{Schur}
\DeclareMathOperator{\Sym}{Sym}

\DeclareMathOperator{\Gal}{Gal} 
\DeclareMathOperator{\GL}{GL} 
\DeclareMathOperator{\Mat}{Mat} 
 
\DeclareMathOperator{\Id}{Id}

\DeclareMathOperator{\rank}{rank} 
\DeclareMathOperator{\nr}{nr} 
\DeclareMathOperator{\dash}{-}

\renewcommand{\epsilon}{\varepsilon}

\title{On admissible tensor products in $p$-adic Hodge theory}
\author{Giovanni Di Matteo}
\email{giovanni.di.matteo@ens-lyon.fr}
\address{UMPA ENS de Lyon, UMR 5669 du CNRS, Universit\'e de Lyon}
\date{\today}

\begin{document}

\begin{abstract}
We prove that if $W$ and $W'$ are two $B$-pairs whose tensor product is crystalline (or semi-stable or de Rham or Hodge-Tate), then there exists a character $\mu$ such that $W(\mu^{-1})$ and $W'(\mu)$ are crystalline (or semi-stable or de Rham or Hodge-Tate).  We also prove that if $W$ is a $B$-pair and $F$ is a Schur functor (for example $\Sym^n(-)$ or $\Lambda^n(-)$) such that $F(W)$ is crystalline (or semi-stable or de Rham or Hodge-Tate) and if the rank of $W$ is sufficiently large, then there is a character $\mu$ such that $W(\mu^{-1})$ is crystalline (or semi-stable or de Rham or Hodge-Tate).  In particular, these results apply to $p$-adic representations.
\end{abstract}

\maketitle

\setlength{\baselineskip}{18pt}
\section*{Introduction} 

Let $K$ and $E$ be two finite extensions of $\Qp$ and let $G_K=\Gal(\Qpbar/K)$. Fontaine has defined the notions of crystalline, semi-stable and de Rham $E$-linear representations of $G_K$ and proved that the corresponding categories are stable under the usual operations (sub-quotient, direct sum and tensor product).  The goal of this note is to answer the following question: if $V$ and $V'$ are two $p$-adic representations whose tensor product is crystalline (or semi-stable or de Rham or Hodge-Tate), then what can be said about $V$ and $V'$?

Berger has defined the tensor category of $B^{\otimes E}_{|K}$-pairs, in which the objects are couples $W=(W_{\e},W_{\dR}^+)$ such that $W_{\e}$ is a $\Be\otimes_{\Qp} E$-representation of $G_K$ and $W_{\dR}^+$ is a $\BdRp\otimes_{\Qp} E$-lattice of $W_{\dR}=(\BdR\otimes_{\Qp} E)\otimes_{(\BdRp\otimes_{\Qp} E)} W_{\e}$.  If $W=(W_{\e},W_{\dR}^+)$ is a  $B^{\otimes E}_{|K}$-pair, then we define the rank of $W$ to be $\rank_{(\Be\otimes_{\Qp} E)}W_{\e}=\rank_{(\BdRp\otimes_{\Qp} E)}W_{\dR}^+$. If $E/\Qp$ is a finite extension and if $V$ is an $E$-linear representation of $G_K$, then $W(V)=((\Be\otimes_{\Qp} E)\otimes_E V, (\BdRp\otimes_{\Qp} E)\otimes_E V)$ is a $B^{\otimes E}_{|K}$-pair, and the functor $W(-)$ identifies the category of $E$-linear representations of $G_K$ with a tensor subcategory of the category of $B^{\otimes E}_{|K}$-pairs. The notions of crystalline, semi-stable, de Rham, and Hodge-Tate objects may be extended in a natural way to objects in the category of $B^{\otimes E}_{|K}$-pairs in such a way that an $E$-linear representation $V$ of $G_K$ is crystalline (or semi-stable or de Rham or Hodge-Tate) if and only if the associated $B^{\otimes E}_{|K}$-pair $W(V)$ is.  

Using Fontaine's theory of $\BdR$-representations (see \cite{Fon04}), we can show the following result.

\begin{ThmHTtensor}
Let $W$ and $W'$ be $B^{\otimes E}_{|K}$-pairs.  If the $B^{\otimes E}_{|K}$-pair $W\otimes W'$ is Hodge-Tate, then there is a finite extension $F/E$ and a character $\mu:G_K\to F^{\times}$ such that the $B^{\otimes F}_{|K}$-pairs $W(\mu^{-1})$ and $W'(\mu)$ are Hodge-Tate.  If, moreover, $W\otimes W'$ is de Rham, then so are $W(\mu^{-1})$ and $W'(\mu)$.
\end{ThmHTtensor}

It is known that every de Rham $B^{\otimes E}_{|K}$-pair is potentially semi-stable, due to the results of \cite{And02}, \cite{Ber02}, \cite{Ked00}, and \cite{Meb02}. The properties of $(\varphi,N,\Gal(L/K))$-modules allow us to understand the situation when $W$ and $W'$ are both potentially semi-stable.

\begin{Thmssttensor}
Let $W$ and $W'$ be $B_{|K}^{\otimes E}$-pairs which are potentially semi-stable.  If the $B_{|K}^{\otimes E}$-pair $W\otimes W'$ is semi-stable, then there is a finite extension $F/E$ and a character $\mu:G_K\to F^{\times}$ such that the $B^{\otimes F}_{|K}$-pairs $W(\mu^{-1})$ and $W'(\mu)$ are semi-stable. If, moreover, $W\otimes W'$ is crystalline, then so are $W(\mu^{-1})$ and $W'(\mu)$.
\end{Thmssttensor}

In particular, the above two theorems may be used to deduce analogous results for $p$-adic representations (see corollaries \ref{dRtensorreps} and \ref{ssttensorreps}).

The same methods used to prove theorems \ref{dRtensor} and \ref{ssttensors} above may be used to understand the situation when the image of a $B$-pair by a Schur functor is crystalline (or semi-stable or de Rham or Hodge-Tate). An integer partition $u=(u_1,\ldots,u_r)\in\mathbf{N}_{>0}^r$ with $u_1\geq\ldots \geq u_r$ of an integer $n$ gives rise to the Schur functor $\Schur^u(-)$, which sends $B^{\otimes E}_{|K}$-pairs to $B^{\otimes E}_{|K}$-pairs. If $u_1=u_2=\ldots =u_r$, then we put $r(u) = r+1$, and we put $r(u)=r$ when this is not the case.

\begin{ThmHTschur}
Let $W$ be a $B^{\otimes E}_{|K}$-pair such that $\rank(W)\geq r(u)$.  If $\Schur^u(W)$ is Hodge-Tate, then there is a finite extension $F/E$ and a character $\mu:G_K\to F^\times$ such that the $B^{\otimes F}_{|K}$-pair $W(\mu^{-1})$ is Hodge-Tate.  If, moreover, $\Schur^u(W)$ is de Rham, then $W(\mu^{-1})$ is de Rham.
\end{ThmHTschur}

\begin{Thmsstschur}
Let $W$ be a potentially semi-stable $B^{\otimes E}_{|K}$-pair such that $\rank(W)\geq r(u)$.  If $\Schur^u(W)$ is semi-stable, then there is a finite extension $F/E$ and a character $\mu:G_K\to F^\times$ such that the $B^{\otimes F}_{|K}$-pair $W(\mu^{-1})$ is semi-stable.  If, moreover, $\Schur^u(W)$ is crystalline, then so is $W(\mu^{-1})$.
\end{Thmsstschur}

In particular, when $u=(n)$, then $r(u)=2$ and the associated Schur functor is $\Sym^n(-)$ and when $u=(1,\ldots,1)$, then $r(u)=n+1$ and the associated Schur functor is $\Lambda^n(-)$.  In the discussion following corollary \ref{HTschurreps}, we show that the bounds on $\rank(W)$ in theorems \ref{dRschur} and \ref{sstschur} are optimal.

It was shown by Skinner (see \S 2.4.1 of \cite{Ski09}) that if $V$ is a $p$-adic representation and if $\Sym^2(V)$ is crystalline, then Wintenberger's methods of \cite{Win95} and \cite{Win97} may be applied to show that there exists a quadratic character $\mu$ such that $V(\mu)$ is crystalline.  It is likely that Wintenberger's methods can be used in the same fashion to give another proof of our theorems \ref{dRtensor}, \ref{ssttensors}, \ref{dRschur}, and \ref{sstschur}.

\vskip 15pt
\noindent\textbf{Acknowledgements.} This note is part of my PhD under the supervision of Laurent Berger. We are grateful to Kevin Buzzard, Frank Calegari, Pierre Colmez, Brian Conrad, Tong Liu and Liang Xiao for useful correspondence and to Jean-Marc Fontaine and Jean-Pierre Wintenberger for pointing out the tannakian argument. Laurent Berger first heard about this question from Barry Mazur.

\section{Notation and generalities}

\subsection{Notation}  
Let $\Qpbar$ be an algebraic closure of $\Qp$ and let $\Cp$ be its $p$-adic completion. Let $\BdR$, $\BdRp$, $\Bcris$, and $\Bst$ denote Fontaine's rings as in \cite{Fon94a} and let $\Be=\Bcris^{\varphi=1}$.  In this note, $E/\Qp$ and $K/\Qp$ denote finite extensions. If $\B$ is any of the above rings, then $\B_E$ will denote the ring $\B\otimes_{\Qp} E$ endowed with an action of $G_K=\Gal(\Qpbar/K)$ defined by $g(b\otimes e) = g(b)\otimes e$ for all $g\in G_K$.  If $W$ is a free $\B_E$-module of finite rank endowed with a semi-linear action of $G_K$, then we refer to $W$ as a $\B_E$-representation of $G_K$.

\subsection{$B_{|K}^{\otimes E}$-pairs}
A {\em $B_{|K}^{\otimes E}$-pair} is a couple $W=(W_{\e},W_{\dR}^+)$ where $W_{\e}$ is a $\BeE{E}$-representation of $G_K$ and $W_{\dR}^+$ is a $G_K$-stable $\BdRpE{E}$-lattice of $W_{\dR}:= (\BdRE{E})\otimes_{(\BeE{E})} W_{\e}$.  We define $\rank(W)$ to be the rank of $W_{\e}$ as a $\BeE{E}$-module. If $W$ and $W'$ are $B_{|K}^{\otimes E}$-pairs, then $W\otimes W'=(W_{\e}\otimes_{\BeE{E}}W'_{\e},W^+_{\dR}\otimes_{\BdRpE{E}}W'^+_{\dR})$ is a $B_{|K}^{\otimes E}$-pair. If $F/E$ and $L/K$ are finite extensions and if $W$ is a $B^{\otimes E}_{|K}$-pair, then $F\otimes_E W|_{G_L}$ is a $B_{|L}^{\otimes F}$-pair.
If $V$ is an $E$-linear representation of $G_K$, then we let $W(V)$ denote the $B_{|K}^{\otimes E}$-pair $\left((\BeE{E})\otimes_E V, (\BdRpE{E})\otimes_E V\right)$.
The properties of $B_{|K}^{\otimes E}$-pairs are developed in \cite{Ber08}, \cite{BerCh}, and \cite{Nak09}.

\subsection{Representations with coefficients in an extension}\label{coefficientsdecomposition}
Let $F/\Qp$ be a finite extension such that $K\supset F^{\Gal}$. If $\B\in\{\Qpbar,\Cp,\BdR\}$, then the map
\begin{equation}\label{descentmap}\addtocounter{subsubsection}{1}
\begin{split}
 \B\otimes_{\Qp} F &\simeq \bigoplus_{h:F\to \Qpbar} \B\\  
 (b\otimes f) &\mapsto (b h(f))_{h}
\end{split}
\end{equation}
 (where $h$ runs over the embeddings of $F$ into $\Qpbar$) is an isomorphism of rings which commutes with the action of $G_K$.

In particular, a $\B_{F}$-representation $W$ of $G_K$ decomposes into a direct sum $W = \bigoplus_{h:F\to \Qpbar} W_h$ as a $\B$-representation of $G_K$, where $W_h$ is the sub-$\B$-representation of $\rank_{\B} W_h=\rank_{\B_F} W$ coming from the $h$-factor map $(b\otimes f)\mapsto b\cdot h(f) : \B\otimes_{\Qp} F\to \B$ of the map \ref{descentmap}.  A $\BdRE{F}$-representation $W$ of $G_K$ is de Rham if and only if the $\BdR$-representations $W_h$ are de Rham for each embedding $h:F\to \Qpbar$ and a $\CpE{F}$-representation $W$ of $G_K$ is Hodge-Tate if and only if the $\Cp$-representations $W_h$ are Hodge-Tate for all embeddings $h:F\to\Qpbar$.

\begin{Lem}\label{dRtensorlem0}
If $W$ and $W'$ are $\B_F$-representations of $G_K$ and if $W=\bigoplus_h W_h$ and $W'=\bigoplus_h W'_h$ are their decompositions as described above, then the decomposition of the $\B_F$-representation $W\otimes_{\B_F} W'$ is given by $\bigoplus_{h:F\to\Qpbar} (W_h\otimes_B W'_h)$.
\end{Lem}

\subsection{Schur functors applied to $B$-pairs}\label{schurnotation}
Let $n\geq 1$ be an integer and let $n=u_1+\ldots +u_r$ be an integer partition such that $u_i\geq u_{i+1}\geq 1$ for all $i\in\{1,\ldots, r-1\}$, which we denote by $u=(u_1,\ldots,u_r)$.  We represent $u$ by its Young diagram $Y_u$, which is a diagram of $n$-many boxes arranged into left-justified rows such that the $i$-th row from the top contains $u_i$-many boxes.  We let $v_j$ denote the length of the $j$-th column from the left. Put $r(u)=r+1$ if $Y_u$ is a rectangle (i.e., if $u_1=\ldots=u_r$) and put $r(u)=r$ if $Y_u$ is not a rectangle.

A \emph{tableau on $Y_u$ with values in $\{1,\ldots, d\}$} is a labeling of the boxes in $Y_u$ with elements in $\{1,\ldots, d\}$ such that the labeling is weakly increasing from left to right and strongly increasing from top to bottom; we let $T=(t_{ij})$ denote a tableau with the integer $t_{ij}\in\{1,\ldots, d\}$ in the $j$-th column of the $i$-th row of $Y_u$.  If $d\geq r$, then there is a tableau on $Y_u$ which has $i$ in each box of the $i$-th row from the top; we refer to this tableau as the standard tableau, and we denote it by $T_1$.  If $d\geq r(u)$, then there are tableau $T_2,\ldots,T_{d}$ on $Y_u$ with values in $\{1,\ldots,d\}$ such that for all $i\in\{1,\ldots, d-1\}$, there is an integer $j\in\{1,\ldots, d-1\}$ such that $T_j$ and $T_{j+1}$ have the same entries in all but one box, and in this box $T_j$ contains $i$ and $T_{j+1}$ contains $i+1$.

Let $R$ be a commutative ring with $1$. The partition $u$ gives rise to the Schur functor $\Schur^u(-)$, which sends $R$-modules to $R$-modules.  
If $M$ is an $R$-module, then $\Schur^u(M)$ may be realized as a quotient of the $R$-module $\Lambda^{v_1}(M)\otimes\ldots\otimes\Lambda^{v_{u_1}}(M)$.  If $\{m_1,\ldots, m_k\}\subset M$ and if $T=(t_{ij})$ is a tableau on $Y_u$ with values in $\{1,\ldots, k\}$, then we let $m_T$ denote the image of the element $(m_{t_{11}}\wedge\ldots\wedge m_{t_{v_1 1}})\otimes\ldots\otimes(m_{t_{1u_1}}\wedge\ldots\wedge m_{t_{v_{u_1}u_1}})$ in $\Schur^u(M)$. If $M$ is a free $R$-module of finite rank with basis $(e_1,\ldots,e_d)$, then $\Schur^u(M)$ is a free $R$-module with basis $(e_T)_{_T}$, where $T$ ranges over all tableaux on $Y_u$ with values in $\{1,\ldots, d\}$.

For example, if $M$ is an $R$-module then the Schur module associated to the partition $u=(n)$ is $\Sym^n(M)$ and the Schur module associated to the partition $u=(1,\ldots,1)$ is $\Lambda^n(M)$.  The fundamental properties of tableaux and Schur modules are developed in \cite{Ful97}.

If $W=(W_{\e},W_{\dR}^+)$ is a $B^{\otimes E}_{|K}$-pair, then $\Schur^u(W)=(\Schur^u(W_{\e}),\Schur^u(W_{\dR}^+))$ is a $B^{\otimes E}_{|K}$-pair.  If $V$ is an $E$-linear representation of $G_K$, then we have an isormorphism of $B^{\otimes E}_{|K}$-pairs $\Schur^u(W(V))\riso W(\Schur^u(V))$.

\begin{Lem}\label{dRschurlem0}
Let $F/\Qp$ be a finite extension such that $K\supset F^{\Gal}$ and let $\B\in\{\Cp,\BdR\}$. If $W$ is a $\B_F$-representation of $G_K$ and if $W= \bigoplus_{h:F\to \Qpbar} W_h$ is the decomposition of $W$ as a $\B$-representation of $G_K$ from paragraph \ref{coefficientsdecomposition}, then the decomposition of the $\B_F$-representation $\Schur^u(W)$ is given by $\Schur^u(W)=\bigoplus_{h:F\to \Qpbar} \Schur^u(W_h)$.
\end{Lem}

\section{Hodge-Tate tensor products and Schur $B$-pairs}

\subsection{Sen's theory of $\Cp$-representations}\label{sentheory}
In this paragraph, we suppose that $E/\Qp$ contains $K^{\Gal}$.
Let $\chi:G_K\to\mbf{Z}_p^\times$ denote the cyclotomic character, $H_K=\Gal(\Qpbar/K_{\infty})$ its kernel, and $\Gamma_K = \Gal(K_{\infty}/K)$. In 
\cite{Sen80}, Sen associates to a $\CpE{E}$-representation $W$ of $G_K$ the $(K_{\infty}\otimes_{\Qp} E)$-module 
$D_{\sen}(W)$, which is free of rank $d=\rank_{\CpE{E}}(W)$ and is endowed with an $K_{\infty}$-semi-linear $E$-linear action of 
$\Gamma_K$, together with a $(K_{\infty}\otimes_{\Qp} E)$-linear operator $\Theta_W$ which gives the action of 
$\text{Lie}(\Gamma_K)$ on $D_{\sen}(W)$.  The action of $\Gamma_K$ commutes with $\Theta_W$, and therefore the characteristic 
polynomial $P_{W}$ has coefficients in $(K_{\infty}\otimes_{\Qp} E)^{\Gamma_K}=K\otimes_{\Qp} E$. If $h:K\to E$ is an embedding, we may associate to $W$ the set of its $h$-weights $\poids^h(W):=\{x\in \Qpbar| P^h_W(x)=0\}$ of
roots of $P^h_W$ counted with multiplicity, where $P^h_W$ is the polynomial of degree $d$ with coefficients in $E$ coming from 
$(e,k)\mapsto e\cdot h(k) : K\otimes_{\Qp} E\to E$. For example, if $\CpE{E}(i)$ denotes the $\CpE{E}$-representation 
associated to the $i$-fold twist by the cyclotomic character ($i\in \mathbf{Z}$) and if $h:K\to E$ is an embedding, then $h$-weight of $\CpE{E}(i)$ is $i$.

In particular, if $W=(W_{\e},W_{\dR}^+)$ is a $B_{|K}^{\otimes E}$-pair, then the above may be applied to the $\CpE{E}$-representation $\ol{W}= W_{\dR}^+/tW_{\dR}^+$.  The set of Sen weights of the underlying $\Qp$-linear representation of $V$ is the set of all $h$-weights for all embeddings $h:K\to E$ together with their images under the embeddings of $E$ into $\Qpbar$.  We let $\poids(W)$ denote the set of all Sen weights 
associated to $V$.  Sen showed in \cite[2.3]{Sen80} that a $\Cp$-representation $W$ is Hodge-Tate if and only if it is semi-simple with integer Sen weights. In particular, an $E$-linear representation $V$ is Hodge-Tate if and only if the $\CpE{E}$-representation $W=\Cp\otimes_{\Qp} V$ is semi-simple and for each $h$, all the $h$-weights of $W$ are in $\mathbf{Z}$.

If the Sen weights of $W$ are all in $\mathbf{Z}$, then \cite[thm. 2.14]{Fon04} implies that $W$ is a direct sum of indecomposible objects of the form $\Cp[i;d] :=\Cp(i)\otimes_{\Zp} \Zp(0;d)$ where $i\in\mathbf{Z}$ is a Sen weight of $W$ and $\Zp(0;d)$ is the $\Zp$-module of polynomials in $\log t$ of degree less than or equal to $d$ with coefficients in $\Zp$. The representation $\Cp[i;d]$ is simple if and only if $d=0$.

The $(K_{\infty}\otimes_{\Qp} E)$-module $D_{\sen}(W)$ and its operator $\Theta_W$ satisfy the following properties.

\begin{Prop}\label{thetas} 
Let $W$ and $W'$ be two $\CpE{E}$-representations of $G_K$. 
\begin{itemize}
\item[(i)] If $W'$ is a sub-representation of $W$, then $\Theta_W |_{W'} = \Theta_{W'}$  and $\Theta_{W/W'}$ is the canonical map induced by $\Theta_W$.  In particular, if $0\to W'\to W\to W''\to 0$ is an exact sequence of $\CpE{E}$-representations, then $P_{\Theta_W}=P_{\Theta_{W'}}P_{\Theta_{W''}}$ and $\poids^h(W) = \poids^h(W')\cup \poids^h(W'')$ (with multiplicity).  

\item[(ii)] If $W$ is a $\CpE{E}$-representation and if $F/E$ is a finite extension, then $D_{\sen}(F\otimes_E W)=F\otimes_E D_{\sen}(W)$ and $\Theta_{F\otimes W}$ is the $F$-linearisation of $\Theta_W$. In particular, the $h$-weights of an $E$-linear representation $V$ are the same as those of $F\otimes_E V$.

\item[(iii)] If $W$ and $W'$ are two $\CpE{E}$-representations of $G_K$, then 
  $D_{\sen}(W\otimes_{\CpE{E}}W')=D_{\sen}(W)\otimes_{(K_{\infty}\otimes_{\Qp} E)} D_{\sen}(W')$ and the Sen operator on 
  $D_{\sen}(W\otimes_{\CpE{E}}W')$ is $\Theta_W\otimes \Id + \Id\otimes \Theta_{W'}$. 
  In particular, for each embedding $h:K\to E$, the $h$-weights of  $W\otimes_{\CpE{E}} W'$ are the elements $s + s'$, where $s$ is an $h$-weight of $W$ and $s'$ is an $h$-weight of $W'$.  
\item[(iv)] If $L/K$ is a finite Galois extension, then $D_{\sen}(W|_{G_L})=L_{\infty}\otimes_{K_\infty} D_{\sen}(W)$ as an $L_{\infty}\otimes_{\Qp} E$-representation of $\Gamma_L$, and $\Theta_{W|_{G_L}}$ is the $L_{\infty}$-linearization of $\Theta_W$.
\end{itemize}
\end{Prop}    

\begin{Cor}\label{Schurweights}
Let $W$ be a $\CpE{E}$-representation of $G_K$. If $h:K\to E$ is an embedding and if $a_{1,h},\ldots, a_{d,h}$ denote the $h$-weights of $W$, then the $h$-weights of $\Schur^u(W)$ are the elements $a_T=\sum_{i,j} a_{t_{ij},h}$ for any tableau $T=(t_{ij})$ on the Young diagram of $u$ with values in $\{1,\ldots, d\}$.
\end{Cor}

\begin{Lem}\label{charwts} 
Let $\omega_{1},\ldots, \omega_{r}$ be elements of $E$ and let $h_1,\ldots,h_r$ be the embeddings of $K$ into $E$.
There exists a finite extension $F/E$ and a character $\mu:G_K\to F^\times$ such that $\poids^{h_i}(F(\mu)) = \{\omega_{i}\}$ for 
$i=1,\ldots,r$.
\end{Lem}

\begin{proof} 
Let $\chi_K:G_K\to \mcl{O}_K^\times$ be the character associated to a Lubin-Tate module over $\mcl{O}_K$. 
The $h$-weight of $K(\chi_K)$ is $1$ if $h$ is the inclusion of $K$ in $E$, and $0$ otherwise \cite[thm. I.2.1]{Col93}.

If $\omega\in E$, then $\omega=p^{-n}\omega'$ for some $\omega'\in \mcl{O}_E$, and some integer $n\geq 0$.  
Consider the topological factorisation $\mcl{O}_K^\times = [k_K^\times]\times (1+\mfk{m}_K)$.  There is a topological 
factorisation of the $\mbf{Z}_p$-module $1+\mfk{m}_K$ into 
$\mbf{Z}/p^a\mbf{Z}\times \mbf{Z}_p^r$, where $a\geq 0$ and $r=[K:\mbf{Q}_p]$.  Let $\langle\chi_K\rangle$ denote the 
projection of $\chi_K$ onto the free submodule $\mbf{Z}_p^r$ in $1+\mfk{m}_K$. If $y_1,\ldots, y_r$ are a 
$\mbf{Z}_p$-basis of $\mbf{Z}_p^r$, and if $F/E$ is an extension containing $z_1,\ldots, z_r\in 1+\mfk{m}_F$ such that 
$z_i^{p^n} = y_i$, then the map $\mu(y_1^{a_1}\cdot\ldots\cdot y_r^{a_r}):= z_1^{\omega' a_1}\cdot\ldots\cdot z_r^{\omega' a_r}$ 
composed with $\langle\chi_K\rangle$ is a character whose $h$-weight is $p^{-n}\omega'=\omega$ when $h=id$ and $0$ otherwise. 
We denote this character by $\langle \chi_K\rangle^{\omega}$.

Given $\omega_{1},\ldots, \omega_{r}\in E$, the product of characters 
$\prod \langle h_i^{-1}(\chi_K)\rangle^{\omega_i}$ has $h_i$-weight equal to $\omega_{i}$ for each $1\leq i\leq r$, where $h_i^{-1} : F \to F$ is the inverse of an automorphism $h_i : F \to F$ extending $h_i : K \to E\subset F$.
\end{proof}
\subsection{Fontaine's theory of $\BdR$-representations}
Let $W$ be a $\BdR$-representation of $G_K$ and let $\mcl{W}$ be a $G_K$-stable $\BdRp$-lattice.  The quotient $\ol{\mcl{W}}:=\mcl{W}/t\mcl{W}$ is a $\Cp$-representation, and we may therefore associate to it the set $\poids(\ol{\mcl{W}})$ of its Sen weights, which is a set of elements of $\Qpbar$ of cardinal $\dim_{\BdR} W$ which is stable by the action of $G_K$.  The following proposition shows that all lattices of $W$ have the same Sen weights up to integers, so that the set of Sen weights  modulo $\mathbf{Z}$ of a lattice $\mcl{W}$ is an invariant of $W$. 

\begin{Prop}
\label{weightintegrality} 
Let $W$ be a $\BdR$-representation of $G_K$.  If $\mcl{W}$ and $\mcl{W}'$ are two $G_K$-stable lattices of $W$, then each Sen weight of $\ol{\mcl{W}'}$ may be written in the form $\alpha + i$ for $\alpha$ a Sen weight $\alpha$ of $\ol{\mcl{W}}$ and $i\in\mathbf{Z}$.
\end{Prop}
\begin{proof}
Let $c\geq 0$ be an integer such that the lattice $t^c\mcl{W}'$ is contained in $\mcl{W}$ and let $c'\geq 0$ be an integer such that the lattice $t^{c'}\mcl{W}$ is contained in $t^c\mcl{W}'$. 

Consider the sequence of $G_K$-stable lattices : 
\[t^c\mcl{W}'=t^c\mcl{W}' + t^{c'}\mcl{W} \subset t^c\mcl{W}' + t^{c'-1}\mcl{W} \subset \ldots
\subset t^c\mcl{W}' + t\mcl{W}	\subset t^c\mcl{W}' + \mcl{W} =\mcl{W},\]
and let $\mcl{X}_{k}$ denote the lattice $t^c\mcl{W}' + t^{c'-k}\mcl{W}$ (for $0\leq k\leq c'$). We have $G_K$-equivariant inclusions $t\mcl{X}_{k+1} \subset \mcl{X}_k \subset \mcl{X}_{k+1}$ for $k=0,1,\ldots,c'-1$ ; we therefore have exact sequences of 
$\Cp$-representations :
\[ \mcl{X}_{k+1}/t\mcl{X}_{k+1} \to \mcl{X}_{k+1}/\mcl{X}_k \to 0 \quad\text{and}\quad  0\to t\mcl{X}_{k+1}/t\mcl{X}_k \to \mcl{X}_{k}/t\mcl{X}_{k} \to \mcl{X}_{k+1}/t\mcl{X}_{k+1} \]
which, taken together with (i) and (iii) of proposition \ref{thetas}, 
and since $x\mapsto tx$ induces an isomorphism of $(\mcl{X}_{k+1}/\mcl{X}_k)(1)$ onto $t\mcl{X}_{k+1}/t\mcl{X}_k$, implies that 
  $\poids(\ol{\mcl{X}_k}) \subset \poids(\ol{\mcl{X}_{k+1}}) \cup (\poids(\ol{\mcl{X}_{k+1}}) + 1)$. 	
  By recurrence, the Sen weights of $\mcl{X}_0=t^c\mcl{W}'$ 
  are	all of the form $\alpha + i$, where $\alpha$ is a Sen weight of $\ol{\mcl{X}_{c'}}=\ol{\mcl{W}}$ and $i$ is an 
  integer.  Again by (iii) of proposition \ref{thetas}, the Sen weights of $\mcl{W}'$ are of the form $\alpha + i$ where  $\alpha$ is a Sen weight of $\ol{\mcl{W}}$.
	\end{proof}

If $W$ is a $\BdR$-representation of $G_K$ and if $\mcl{W}\subset W$ is a $G_K$-stable lattice, we call the image of the set $\poids(\ol{\mcl{W}})$ modulo $\mathbf{Z}$ the set of {\em de Rham weights} of $W$, and we denote this set by $\drpoids(W)$. The set of de Rham weights of $W$ is endowed with an action of $G_K$. Fontaine's theorem \cite[3.19]{Fon04} shows that any $\BdR$-representation $W$ decomposes along the set of $G_K$-orbits in $\drpoids(W)$, and that $W$ is de Rham if and only if it is semi-simple with de Rham weights in $\Z$.  

If the de Rham weights of $W$ are all in $\mathbf{Z}$, then Fontaine's theorem \cite[3.19]{Fon04} implies that $W$ is a direct sum of indecomposible objects of the form $\BdR[\{0\};d] :=\BdR\otimes_{\Zp} \Zp(0;d)$ where $\Zp(0;d)$ is the $\Zp$-module of polynomials in one variable $X=\log t$ of degree less than or equal to $d$ with coefficients in $\Zp$, such that $g(X)=X + \log(\chi(g))$ for all $g\in G_K$. The representation $\BdR[\{0\};d]$ is simple if and only if $d=0$.

\subsection{Hodge-Tate and de Rham tensor products of $B$-pairs}\label{HTdRpar}

\begin{Lem}\label{dRtensorlem1}
If $W$ and $W'$ are $\Cp$-representations of $G_K$ with Sen weights in $\Z$ such that $W\otimes_{\Cp} W'$ is Hodge-Tate, then $W$ and $W'$ are Hodge-Tate.

If $W$ and $W'$ are $\BdR$-representations of $G_K$ with de Rham weights in $\Z$ such that $W\otimes_{\BdR} W'$ is de Rham, then $W$ and $W'$ are de Rham.  
\end{Lem}
\begin{proof}
Let $W$ and $W'$ be $\BdR$-representations of $G_K$ with de Rham weights in $\Z$. By Fontaine's theorem \cite[3.19]{Fon04}, $W$ and $W'$ admit unique decompositions $W\simeq\bigoplus_{i=1}^r \BdR[\{0\};d_i]^{e_i}$ and $W'\simeq\bigoplus_{j=1}^{r'} \BdR[\{0\}; d'_j]^{e'_j}$. The $\BdR$-representations $W$ and $W'$ are de Rham if and only if all of the $d_i$ and $d'_j$ are equal to zero. If $W\otimes_{\BdR}W'$ is de Rham, then $\BdR[\{0\};d_i]\otimes_{\BdR}\BdR[\{0\}; d'_j]$ is de Rham for every $1\leq i\leq r$ and $1\leq j\leq r'$.  Suppose, for example, that $W$ is not de Rham, so that we may assume $d_1>0$.  Let $U=\BdR[\{0\};d_1]\otimes_{\BdR}\BdR[\{0\};d'_1]$, let $e_1=1\otimes1$, and let $(e_1, e_2,\ldots, e_{f})$ be a $K$-basis of $D_{\dR}(U)=(\BdR\otimes_{\Qp} U)^{G_K}$, where $f=(d_1+1)(d'_1+1)$.  If $U$ is de Rham, then the element $X\otimes 1\in U$ may be written as a sum $X\otimes 1 = b_1(1\otimes 1) + \sum_{i=2}^{f} b_ie_i$ with $b_i\in\BdR$ for all $1\leq i\leq f$. Since $g(X\otimes 1) = X\otimes 1 + \log(\chi(g))(1\otimes 1)$ for all $g\in G_K$, we have $g(b_1) - b_1 = \log(\chi(g))$ for all $g\in G_K$.  If $b_1\in \BdRp$ then $g(\theta(b_1)) - \theta(b_1) = \log\chi(g)$ for all $g\in G_K$, which is impossible since $g\mapsto \log\chi(g)$ is a generator of the one-dimensional $K$-vector space $H^1(G_K,\Cp)$.  If $b_1\in t^h\BdRp$ for some $h<0$, then $b_1=t^hb'$ with $b'\in \BdRp\backslash t\BdRp$ and $\chi(g)^hg(b') -b'\in t^{-h}\BdRp\subset t\BdRp$, so that reducing modulo $t$ would imply that $\theta(b')\in \Cp(h)^{G_K}=\{0\}$, a contradiction. We therefore see that $W$ and $W'$ must be de Rham.

The same arguments together with Fontaine's theorem \cite[2.14]{Fon04} show that if $W$ and $W'$ are $\Cp$-representations with Sen weights in $\Z$ such that $W\otimes_{\Cp} W'$ is Hodge-Tate, then $W$ and $W'$ are Hodge-Tate.
\end{proof}

\begin{Thm}\label{dRtensor}
Let $W$ and $W'$ be $B^{\otimes E}_{|K}$-pairs.  If the $B^{\otimes E}_{|K}$-pair $W\otimes W'$ is Hodge-Tate, then there is a finite extension $F/E$ and a character $\mu:G_K\to F^{\times}$ such that the $B^{\otimes F}_{|K}$-pairs $W(\mu^{-1})$ and $W'(\mu)$ are Hodge-Tate.  If, moreover, $W\otimes W'$ is de Rham, then so are $W(\mu^{-1})$ and $W'(\mu)$.
\end{Thm}
\begin{proof}
Let $W$ and $W'$ be $B^{\otimes E}_{|K}$-pairs and suppose that the $B^{\otimes E}_{|K}$-pair $W\otimes W'$ is Hodge-Tate.  By extending scalars if necessary, we may suppose that $E$ is finite Galois and contains $K$, so that the methods of paragraph \ref{sentheory} apply.

Let $r=\rank(W)$ and let $r'=\rank(W')$. For each embedding $h:K\to E$, let $a_{1,h},\ldots, a_{r,h}$ denote the $h$-weights of the $\CpE{E}$-representation $\ol{W}$ and let $a'_{1,h},\ldots, a'_{r',h}$ denote the $h$-weights of $\ol{W'}$ as in paragraph \ref{sentheory}.  Part (iii) of proposition \ref{thetas} implies that if $h:K\to E$ is an embedding, then the $h$-weights of $\ol{W\otimes W'}$ are the elements $a_{i,h} + a'_{j,h}$ for $1\leq i\leq r$ and $1\leq j\leq r'$, which are integers since the $\CpE{E}$-representation $\ol{W\otimes W'} = \ol{W}\otimes_{\CpE{E}} \ol{W'}$ is Hodge-Tate. By lemma \ref{charwts}, there is a finite extension $F/E$ and a character $\mu:G_K\to F^\times$ such that for all embeddings $h:K\to E\subset F$, the $h$-weight of the $\CpE{F}$-representation $\ol{W(F(\mu))}$ is $a_{1,h}$.  We may suppose that $F/E$ is Galois.

We now show that the $B^{\otimes F}_{|K}$-pairs $W(\mu^{-1})$ and $W'(\mu)$ are Hodge-Tate.  If $h:K\to E\subset F$ is an embedding, then (ii) and (iii) of proposition \ref{thetas} imply that the $h$-weights of $W(\mu^{-1})$ are the integers $a_{i,h}-a_{1,h}$ (for $1\leq i\leq r$) and the $h$-weights of $W'(\mu)$ are the integers $a_{1,h} + a'_{j,h}$ for $1\leq j\leq r'$.  Since being Hodge-Tate is the same as being potentially Hodge-Tate, it suffices to show that the $B^{\otimes F}_{|F}$-pairs $W(\mu^{-1})|_{G_F}$ and $W'(\mu)|_{G_F}$ are Hodge-Tate.  Let $\ol{W(\mu^{-1})} = \bigoplus_{h:F\to F} \ol{W(\mu^{-1})}_h$ and $\ol{W'(\mu)} = \bigoplus_{h:F\to F} \ol{W'(\mu)}_h$ be the decompositions of $\CpE{F}$-representations of $G_F$ as described in paragraph \ref{coefficientsdecomposition}. The $\Cp$-representations $\ol{W(\mu^{-1})}_h$ and $\ol{W'(\mu)}_h$ have weights in $\Z$ for every $h$. The isomorphism \[\ol{W(\mu^{-1})\otimes W'(\mu)}\simeq \bigoplus_{h:F\to F}  \ol{W(\mu^{-1})}_h\otimes_{\Cp} \ol{W'(\mu)}_h\] of $\Cp$-representations of $G_F$ as in lemma \ref{dRtensorlem0} implies that $\ol{W(\mu^{-1})}_h\otimes_{\Cp} \ol{W'(\mu)}_h$ is Hodge-Tate for each embedding $h:F\to F$. By lemma \ref{dRtensorlem1}, $\ol{W(\mu^{-1})}_h$ and $\ol{W'(\mu)}_h$ are Hodge-Tate for each embedding $h:F\to F$, and therefore $\ol{W(\mu^{-1})}$ and $\ol{W'(\mu)}$ are Hodge-Tate. Therefore, the $B^{\otimes F}_{|K}$-pairs $W(\mu^{-1})$ and $W'(\mu)$ are Hodge-Tate.

Suppose now that $W$ and $W'$ are $B^{\otimes E}_{|K}$-pairs and that the $B^{\otimes E}_{|K}$-pair $W\otimes W'$ is de Rham. By the above, there is a finite extension $F/E$ and a character $\mu:G_K\to F^\times$ such that the $B^{\otimes F}_{|K}$-pairs $W(\mu^{-1})$ and $W'(\mu)$ are Hodge-Tate.  We now show that $W(\mu^{-1})$ and $W'(\mu)$ are de Rham. It suffices to show that the restrictions of $W(\mu^{-1})$ and $W'(\mu)$ to $G_F$ are de Rham. Let $W(\mu^{-1})_{\dR}= \bigoplus_{h:F\to F} W(\mu^{-1})_{\dR,h}$ and $W'(\mu)_{\dR} = \bigoplus_{h:F\to F} W'(\mu)_{\dR,h}$ be the decompositions of $\BdR$-representations of $G_F$ as in paragraph \ref{coefficientsdecomposition}. For each embedding $h:F\to F$, the $\BdR$-representations $W(\mu^{-1})_{\dR,h}$ and $W'(\mu)_{\dR,h}$ have de Rham weights in $\Z$. By lemma \ref{dRtensorlem0}, the $\BdR$-representation $W(\mu^{-1})_{\dR,h}\otimes_{\BdR} W'(\mu)_{\dR,h}$ is de Rham for each embedding $h:F\to F$, and therefore so are $W(\mu^{-1})_{\dR,h}$ and $W'(\mu)_{\dR,h}$ by lemma \ref{dRtensorlem1}.  Therefore, the $B^{\otimes F}_{|K}$-pairs $W(\mu^{-1})$ and $W'(\mu)$ are de Rham.
\end{proof}

\begin{Cor}\label{dRtensorreps}
Let $E/\Qp$ and $K/\Qp$ be finite extensions, and let $V$ and $V'$ be two $E$-linear representations of $G_K$.
If $V\otimes_E V'$ is Hodge-Tate, then there is a finite extension $F/E$ and a character $\mu:G_K\to F^\times$ such that $V(\mu^{-1})$ and $V'(\mu)$ are Hodge-Tate. If $V\otimes_E V'$ is de Rham, then so are $V(\mu^{-1})$ and $V'(\mu)$.
\end{Cor}

\subsection{Hodge-Tate and de Rham Schur $B$-pairs}
In what follows, let $n\geq 1$ be an integer and let $u=(u_1,\ldots,u_r)$ denote an integer partition $n=u_1+\ldots + u_r$ ($u_i\geq u_{i+1}\geq 1$) of $n$.  If $u_1=\ldots = u_r$, put $r(u) = r+1$. Otherwise, put $r(u)=r$.

\begin{Lem}\label{dRschurlem1}
If $W$ is a $\Cp$-representation of $G_K$ having Sen weights in $\Z$ such that $\dim_{\Cp}(W)\geq r(u)$ and $\Schur^u(W)$ is Hodge-Tate, then $W$ is Hodge-Tate.

If $W$ is a $\BdR$-representation of $G_K$ having de Rham weights in $\Z$ such that $\dim_{\BdR}(W)\geq r(u)$ and $\Schur^u(W)$ is de Rham, then $W$ is de Rham.
\end{Lem}
\begin{proof}
Let $W$ be a $\BdR$-representation of $G_K$ having de Rham weights in $\Z$ such that $\dim_{\BdR}(W)\geq r(u)$.  If $W$ is not de Rham, then Fontaine's theorem \cite[3.19]{Fon04} gives a decomposition $W=\BdR[\{0\};d]\oplus W'$ for some $d>0$, so that 
\[\Schur^u(W) \simeq \bigoplus_{\lambda,\mu} (\Schur^{\lambda}(\BdR[\{0\};d])\otimes_{\BdR} \Schur^{\mu}(W'))^{\oplus c_{\lambda,\mu}^u}\]
as a $\BdR$-representation of $G_K$, where $c_{\lambda,\mu}^u \geq 0$ denotes the Littlewood-Richardson number. There are $\lambda$ and $\mu$ such that $\Schur^{\lambda}(\BdR[\{0\};d])\otimes_{\BdR} \Schur^{\mu}(W')$ is non-zero, and such that $d+1\geq r(\lambda)$. This can be seen by using the fact that $c_{\lambda,\mu}^u$ is equal to the number of pairs of tableaux $T$ of shape $\lambda$ and $U$ of shape $\mu$ such that the product tableau $T\cdot U$ is equal to the standard tableau $T_1$ on the Young diagram of $u$.

  The $\BdR$-representations $\Schur^{\lambda}(\BdR[\{0\};d])$ and $\Schur^{\mu}(W')$ have de Rham weights in $\Z$ by lemma \ref{thetas}.  If $\Schur^u(W)$ is de Rham, then so is $\Schur^{\lambda}(\BdR[\{0\};d])\otimes_{\BdR} \Schur^{\mu}(W')$ and lemma \ref{dRtensorlem1} implies that $\Schur^{\lambda}(\BdR[\{0\};d])$ is de Rham.  Let $(1,X,X^2,\ldots,X^d)$ denote the standard $\BdR$-basis of $\BdR[\{0\};d]$.  If $T$ is the tableau having $i$ in the $i$-th row, then the element $e_T\in \Schur^\lambda(\BdR[\{0\};d])$ is such that $g(e_T) = e_T$ for all $g\in G_K$. Let $T'$ be the tableau with values in $\{1,\ldots, d+1\}$ which is obtained from $T$ by adding 1 to the value in the bottom-most cell of the right-most column; this tableau $T'$ exists since $d+1\geq r(\lambda)$.  A calculation shows that $g(e_{T'}) = e_{T'} + \nu\log\chi(g) e_{T}$, where $\nu$ is the length of the right-most column.  If $\Schur^\lambda(\BdR[\{0\};d])$ is de Rham, then it admits a basis $(e_T,e_2,\ldots, e_f)$ of elements such that for all $i=2,\ldots, f$, $g(e_i)=e_i$ for all $g\in G_K$.  If $b_1,\ldots, b_f\in \BdR$ are elements such that $e_{T'}=b_1e_T + \sum_{i\geq 2} b_ie_i$, then $g(b_1) - b_1 = \nu\log\chi(g)$ for all $g\in G_K$, which is impossible.  Therefore, $W$ and $W'$ must be de Rham.

One can prove the claim for $\Cp$-representations by using Fontaine's theorem \cite[2.14]{Fon04} and applying the same arguments.
\end{proof}

\begin{Thm}\label{dRschur}
Let $W$ be a $B^{\otimes E}_{|K}$-pair such that $\rank(W)\geq r(u)$.  If the $B^{\otimes E}_{|K}$-pair $\Schur^u(W)$ is Hodge-Tate, then there is a finite extension $F/E$ and a character $\mu:G_K\to F^\times$ such that the $B^{\otimes F}_{|K}$-pair $W(\mu^{-1})$ is Hodge-Tate.  If, moreover, $\Schur^u(W)$ is de Rham, then $W(\mu^{-1})$ is de Rham.
\end{Thm}
\begin{proof}
Let $W$ be a $B^{\otimes E}_{|K}$-pair such that $d=\rank(W)\geq r(u)$ and suppose that $\Schur^u(W)$ is Hodge-Tate. By extending scalars if necessary, we may suppose that $E/\Qp$ is finite Galois and contains $K$. 

If $h:K\to E$ is an embedding, then let $a_{1,h},\ldots, a_{d,h}$ denote the $h$-weights of $\ol{W}$. By corollary \ref{Schurweights}, the $h$-weights of the $\CpE{E}$-representation $\ol{\Schur^u(W)}=\Schur^u(\ol{W})$ are the elements of the form $a_{T,h}=\sum a_{t_{ij},h}$ for any tableau $T=(t_{ij})$ with values in $\{1,\ldots, d\}$ on the Young diagram of $u$.  Since $\Schur^u(W)$ is Hodge-Tate, the elements $a_T$ are in $\Z$.  Since $d=\rank(W)\geq r(u)$, considering the tableaux $T_1,\ldots,T_d$ in paragraph \ref{schurnotation} allows us to conclude that $a_{i,h} - a_{1,h}\in \Z$ for all $1\leq i\leq d$.  By lemma \ref{charwts}, there is a finite Galois extension $F/E$ and a character $\mu:G_K\to F^\times$ such that the $B^{\otimes F}_{|K}$-pair $W(F(\mu))$ has $a_{1,h}$ has its $h$-weight for each embedding $h:K\to E\subset F$.

We now show that the $B^{\otimes F}_{|K}$-pair $W(\mu^{-1})$ is Hodge-Tate. It suffices to show that the restriction of $W(\mu^{-1})$ to $G_F$ are Hodge-Tate. Let $\ol{W(\mu^{-1})}=\bigoplus_{h:F\to F} \ol{W(\mu^{-1})}_h$ be the decomposition as a $\Cp$-representation of $G_F$ as described in paragraph \ref{coefficientsdecomposition}.  The $\Cp$-representation $\ol{W(\mu^{-1})}_h$ has Sen weights in $\Z$ for each embedding $h:F\to F$. By lemma \ref{dRschurlem0}, the $\Cp$-representation $\Schur^u(\ol{W(\mu^{-1})}_h)$ of $G_F$ is Hodge-Tate for each embedding $h:F\to F$. Since $\dim_{\Cp} \ol{W(\mu^{-1})}_h= \rank(W) \geq r(u)$, lemma \ref{dRschurlem1} implies that $\ol{W(\mu^{-1})}_h$ is Hodge-Tate for each embedding $h:F\to F$. The $B^{\otimes F}_{|K}$-pair $W(\mu^{-1})$ is therefore Hodge-Tate.

Suppose now that $W$ is a $B^{\otimes E}_{|K}$-pair such that $\rank(W)\geq r(u)$ and $\Schur^u(W)$ is de Rham. There is a finite extension $F/E$ and a character $\mu:G_K\to F^\times$ such that the $B^{\otimes E}_{|K}$-pair $W(\mu^{-1})$ is Hodge-Tate. We now show that $W(\mu^{-1})$ is de Rham. Let $W(\mu^{-1})_{\dR}\simeq \bigoplus_{h:F\to F} W(\mu^{-1})_{\dR,h}$ be the decomposition as a $\BdR$-representation of $G_F$ as described in paragraph \ref{coefficientsdecomposition}. The $\BdR$-representation $W(\mu^{-1})_{\dR,h}$ has de Rham weights in $\Z$ for each embedding $h:F\to F$.  By lemma \ref{dRschurlem0}, the $\Schur^u( W(\mu^{-1})_{\dR,h})$ is a de Rham $\BdR$-representation of $G_F$ for each embedding $h:F\to F$ and therefore $W(\mu^{-1})_{\dR,h}$ is de Rham for each embedding $h$ since $\dim_{\BdR} W(\mu^{-1})_{\dR,h}= \rank(W) \geq r(u)$.  Therefore, the $B^{\otimes F}_{| K}$-pair $W(\mu^{-1})$ is de Rham.
\end{proof}

\begin{Cor}\label{HTschurreps}
Let $n\geq 1$ be an integer, let $u$ be a partition of $n$, and let $V$ be an $E$-linear representation of $G_K$ such that $\dim_E(V)\geq r(u)$.  If $\Schur^u(V)$ is Hodge-Tate, then there is a finite extension $F/E$ and a character $\mu:G_K\to F^\times$ such that $V(\mu^{-1})$ is Hodge-Tate. If, moreover, $\Schur^u(V)$ is de Rham, then $V$ is de Rham.
\end{Cor}

We now show that the bound on $\rank(W)$ in theorem \ref{dRschur} is optimal.  If $W$ is a $B^{\otimes E}_{|K}$-pair such that $\rank(W)<r(u)$ then $\Schur^u(W)$ is of rank $1$ if $u_1=\ldots =u_r$ and $\Schur^u(W)=0$ otherwise. In the former case, $\rank(W)=r$ and $\Schur^u(W)=\bigotimes_{i=1}^r \det(W)$.  Let $V$ denote a $2$-dimensional $\Qp$-vector space endowed with an action of $G_{\Qp}$ such that $g\in G_{\Qp}$ acts on a basis $\mcl{E}=(e_1,e_2)$ by the matrix \[\left(\begin{array}{cc}1 & \log_p(\chi(g))\\ 0 & 1 \end{array}\right)\]
so that $V$ is not Hodge-Tate since $\Cp\otimes_{\Qp} V = \Cp[\{0\};1]$, but $G_{\Qp}$ acts trivially on $\Lambda^2V$. There is no character $\mu:G_{\Qp}\to E^\times$ such that $V(\mu^{-1})$ is Hodge-Tate; such a character would necessarily have weights in $\mathbf{Z}$, and lemma \ref{dRschurlem1} would imply that $V$ itself is Hodge-Tate.

\section{Semi-stable tensor products and Schur $B$-pairs}

\subsection{Semi-stable $B$-pairs} 
Let $W=(W_{\e}, W_{\dR}^+)$ be a $B_{|K}^{\otimes E}$-pair.  We say that $W$ is \emph{crystalline} if the $\Bcris$-representation
$(\BcrisE{E})\otimes_{\BeE{E}} W_{\e}$ of $G_K$ is trivial.  Similarly, we say that $W$ is \emph{semi-stable} if the $\Bst$-representation $(\BstE{E})\otimes_{\BeE{E}} W_{\e}$ of $G_K$ is trivial.  We say that $W$ is \emph{potentially crystaline} (or \emph{potentially semi-stable}) if there is a finite extension
$L/K$ such that the $B_{|L}^{\otimes E}$-pair $W|_{G_L}$ is crystalline (or semi-stable).  Note that if $V$ is an $E$-linear representation of $G_K$, then $V$ is crystalline (or semi-stable) if and only if the $B^{\otimes E}_{|K}$-pair $W(V)$ is crystalline (or semi-stable).

Let $L/K$ be a finite Galois extension and let $L_0=L\cap\mathbf{Q}_p^{\nr}$. If $W$ is a $B_{|K}^{\otimes E}$-pair which is semi-stable when restricted to $G_L$, then $D_{\st,L}(W)=(\BstE{E}\otimes_{\mathbf{B}_{\e,E}} W_{\e})^{G_L}$ is a free $L_0\otimes_{\Qp} E$-module such that $\rank_{L_0\otimes_{\Qp} E}(D_{\st,L}(W))=\rank(W)$, endowed with an injective additive self-map $\varphi$ that is $E$-linear and semi-linear for the absolute Frobenius automorphism $\sigma$ on $L_0$, an $L_0\otimes_{\Qp} E$-linear nilpotent endomorphism $N$ such that $N\varphi = p\varphi N$, and an $E$-linear and $L_0$-semi-linear action of $\Gal(L/K)$ which commutes with $\varphi$ and $N$.  The following follows from \cite[4.2.6, 5.1.5]{Fon94b}. 

\begin{Prop}\label{sstreps}
Let $W$ be a potentially semi-stable $B_{|K}^{\otimes E}$-pair, semi-stable when restricted to $G_L$ where $L/K$ is finite and Galois. The  $B_{|K}^{\otimes E}$-pair $W$ is semi-stable if and only if the inertia group $I_{L/K}$ acts trivially on $D_{\st,L}(W)$, and $W$ is crystalline if and only if it is semi-stable 
and $N=0$ on $D_{\st,L}(W)$.
\end{Prop}

\subsection{Semi-stable tensor products}
\begin{Thm}\label{ssttensors}
Let $W$ and $W'$ be two $B_{|K}^{\otimes E}$-pairs which are potentially semi-stable.  If the $B_{|K}^{\otimes E}$-pair $W\otimes W'$ is semi-stable, then there is a finite extension $F/E$ and a character $\mu:G_K\to F^{\times}$ such that $B^{\otimes F}_{|K}$-pairs $W(\mu^{-1})$ and $W'(\mu)$ are semi-stable. If, moreover, $W\otimes W'$ is crystalline, then so are $W(\mu^{-1})$ and $W'(\mu)$.
\end{Thm}
\begin{proof}
Let $L/K$ be a finite Galois extension such that $W$ and $W'$ are semi-stable as $B_{|L}^{\otimes E}$-pairs. By \cite[5.1.7]{Fon94b}, we have an isomorphism of $E-(\varphi,N,\Gal(L/K))$-modules: \[D_{\st,L}(W\otimes W')\liso D_{\st,L}(W)\otimes_{(L_0\otimes_{\Qp} E)} D_{\st,L}(W'). \]

Let $\mcl{E}\subset D_{\st,L}(W)$ and $\mcl{E}'\subset D_{\st,L}(W')$ be $L_0\otimes_{\Qp} E$-bases, so that the set $\mcl{E}\otimes\mcl{E}'$ of elementary tensors is a basis of $D_{\st,L}(W\otimes W')$.  For all $g\in G_K$, let $U_g=\Mat(g|\mcl{E})\in\GL_d(L_0\otimes_{\Qp} E)$ and let 
$U'_g=\Mat(g|\mcl{E}')\in\GL_{d'}(L_0\otimes_{\Qp} E)$. By proposition \ref{sstreps}, $I_{L/K}$ acts trivially on 
$D_{\st,L}(W\otimes_E W')$, and we have $\Mat(g|\mcl{E}\otimes\mcl{E}')=U_g\otimes U'_g = \text{Id}$ for all $g\in I_{L/K}$, so that $U_g = \eta_g\Id$ and $U'_g = \eta^{-1}_g\Id$ with $\eta_g\in (L_0\otimes_{\Qp} E)^\times$. The relation $\varphi g = g\varphi$ on $D_{\st,L}(W)$ translates to the matrix relation $\Mat(\varphi|\mcl{E})\cdot \sigma(U_g) = U_g\cdot g(\Mat(\varphi|\mcl{E}))$ for all 
$g\in\Gal(L/K)$, so that for all $g\in I_{L/K}$, we have $\eta_g \in (L_0\otimes_{\Qp} E)^{\sigma=1}=E$ and therefore $\eta_g \in E^\times$. 

We now show that there is a finite extension $F/E$ such that the character $\eta:I_{L/K}\to E^\times$ can be extended to a character $\mu:\Gal(L/K)\to F^\times$. Let $\omega\in \Gal(L/K)$ be such that its residual image generates the cyclic group 
$\Gal(k_L/k_K)$.  If $g\in \Gal(L/K)$, then we can write $g=g'\omega^i$ for a unique $g'\in I_{L/K}$ and unique 
$0\leq i\leq f-1$, where $f=[k_L:k_K]$. Let $\xi\in\Qpbar$ be an $f^{th}$ root of $\eta(\omega^f)$. Since
$\eta(\omega g' \omega^{-1}) = \eta(g')$ for all $g'\in I_{L/K}$, putting $F=E(\xi)$ and $\mu(g) := \eta(g')\xi^i$ defines a homomorphism 
$\mu:G_K\to F^\times$.

The $B^{\otimes F}_{|K}$-pairs $W(\mu^{-1})$ and $W'(\mu)$ are semi-stable, by proposition \ref{sstreps}.  If, moreover, 
$W\otimes W'$ is crystalline, then the $B^{\otimes F}_{|K}$-pair $W(\mu^{-1})\otimes W'(\mu)$ is crystalline as well
and by the isomorphism of $F-(\varphi,N,\Gal(L/K))$-modules recalled above, we have : 
  $$D_{st,L}(W(\mu^{-1})\otimes W'(\mu))\liso D_{st,L}(W(\mu^{-1}))\otimes_{(L_0\otimes_{\Qp} F)} D_{st,L}(W'(\mu)).$$
The monodromy operator $N\otimes\Id + \Id\otimes N'$ is zero, and therefore the matrices of $N$
and $N'$ are scalar multiples of the identity. Since $N$ and $N'$ are nilpotent, these scalars are necessarily zero 
since $L_0\otimes_{\Qp} F$ is reduced, and thus $W(\mu^{-1})$ and $W'(\mu)$ are crystalline by \ref{sstreps}.
\end{proof}

\begin{Cor}\label{ssttensorreps}
Let $V$ and $V'$ be potentially semi-stable $E$-linear representations of $G_K$. 
If $V\otimes_E V'$ is semi-stable, then there is a finite extension $F/E$ and a character $\mu:G_K\to F^\times$ such that the $F$-linear representations $V(\mu^{-1})$ and $V'(\mu)$ are semi-stable. If, moreover, $V\otimes_E V'$ is crystalline, then so are $V(\mu^{-1})$ and $V'(\mu)$.  
\end{Cor}

\subsection{Semi-stable Schur representations}

In this paragraph, $n\geq 1$ is an integer and $u=(u_1,\ldots, u_r)$ denotes an integer partition $n=u_1+\ldots+u_r$ such that $u_i\geq u_{i+1}\geq 1$.

\begin{Lem}\label{lemabove}
Let $L/K$ be a finite Galois extension and let $D$ be an $E\dash(\varphi,N,\Gal(L/K))$-module such that $\rank(D)\geq r(u)$.  If $I_{L/K}$ acts trivially on $\Schur^u(D)$, then $I_{L/K}$ acts on $D$ via a character $\mu:I_{L/K}\to E^\times$.  If $N=0$ on $\Schur^u(D)$, then $N=0$ on $D$.
\end{Lem}
\begin{proof} 
By extending scalars if necessary, we may suppose that $E\supset L$.  We have an isomorphism of rings, 
$L_0\otimes_{\Qp} E\riso \bigoplus_{h:L_0\to\Qpbar} E$
on which $I_{L/K}$ acts trivially on both sides since $L_0/\Qp$ is Galois.  We therefore see that $D$ decomposes as an $E$-linear representation of $I_{L/K}$ into $D\simeq \bigoplus_h D_h$ where $D_h$ is the $E$-linear representation of $I_{L/K}$ coming from the $h$-factor map $(\lambda,e)\mapsto h(\lambda)e:L_0\otimes_{\Qp} E\to E$. The corresponding decomposition of $\Schur^u(D)$ is given by $\Schur^u(D)\simeq \bigoplus_h\Schur^u(D_h)$, and by assumption $I_{L/K}$ acts trivially on each $E$-linear representation $\Schur^u(D_h)$.  Let $I_{L/K}$ act $\Qpbar$-linearly on $\ol{D}_h=\Qpbar\otimes_E D_h$.  Let $g\in I_{L/K}$. Since $I_{L/K}$ is finite, there is a $\Qpbar$-basis $\mcl{E}^g_h=(e_{1,h}^g,\ldots, e_{d,h}^g)$ of $\ol{D}_h$ and elements $\lambda_{1,h}^g,\ldots, \lambda_{d,h}^g\in\Qpbar$ such that $g(e_{i,h}^g)=\lambda_{i,h}^ge_{i,h}^g$ for all $i\in\{1,\ldots, d\}$.  Consider the $\Qpbar$-basis of $\Schur^u(\ol{D}_h)$ consisting of elements $e_{T,h}^g$, where $T$ ranges over all tableaux on $Y_u$. One has $g(e_{T,h}^g)=\lambda_{T,h}^ge_{T,h}^g$, where $\lambda_{T,h}=\prod_{i=1}^d \lambda_{i,h}^{m_T(i)}$ and $m_T(i)$ denotes the number of times that $i$ appears in the tableau $T$. Since $\dim_{\Qpbar} \ol{D}_h = \rank(D) \geq r(u)$, one sees that $\lambda_{1,h}^g=\lambda_{2,h}^g=\ldots=\lambda_{d,h}^g=\lambda^g$ by considering the tableaux $T_1,\ldots, T_d$ as in paragraph \ref{schurnotation}, so that $g(z)=\lambda^g_hz$ for all $z\in \ol{D}_h$.  Note that we necessarily have $\lambda^g_h\in E$.  We therefore see that for each embedding $h:L_0\to E$, $I_{L/K}$ acts on $\ol{D}_h$ by a character $\mu_h:I_{L/K}\to E^\times$, which translates to saying that $I_{L/K}$ acts on $\ol{D}$ by a character $\mu:I_{L/K}\to (L_0\otimes_{\Qp} E)^\times$. Since $\varphi g = g\varphi$ for all $g\in I_{L/K}$ and $(L_0\otimes_{\Qp} E)^{\sigma=1}=E$, we see that $\mu:I_{L/K}\to E^\times$.

Moreover, since $N$ is an $(L_0\otimes_{\Qp} E)$-linear map, the factors in the decomposition $D\simeq \bigoplus_h D_h$ are $N$-stable. We let $N$ again denote the $E$-linear nilpotent map induced on $D_h$. Since $N=0$ on $\Schur^u(\ol{D})=\bigoplus_h \Schur^u(\ol{D}_h)$, we see that $N=0$ on $\Schur^u(\ol{D}_h)$ for each embedding $h:L_0\to \Qpbar$.  Let $(e'_{1,h},\ldots, e'_{d,h})$ denote a Jordan canonical basis for $N$ on $\ol{D}_h$.  Suppose that $N\neq 0$, so that we may suppose $N(e'_{2,h}) =e'_{1,h}$. If $T$ is the tableau on $Y_u$ in which $i$ appears in all boxes of the $i$-th row, except in the right most column where $i+1$ appears, then a calculation shows that $N(e_{T,h})=e_{T',h}$, where $T'$ is another tableau, therefore contradicting the fact that $N=0$ on $\ol{D}_h$. We therefore see that $N=0$ on each $\ol{D}_h$, so that $N=0$ on $\ol{D}$ and thus $N=0$ on $D$.
\end{proof}

\begin{Thm}\label{sstschur}
Let $W$ be a potentially semi-stable $B^{\otimes E}_{|K}$-pair such that $\rank(W)\geq r(u)$.  If the $B^{\otimes E}_{|K}$-pair $\Schur^u(W)$ is semi-stable, then there is a finite extension $F/E$ and a character $\mu:G_K\to F^\times$ such that the $B^{\otimes F}_{|K}$-pair $W(\mu^{-1})$ is semi-stable.  If, moreover, $\Schur^u(W)$ is crystalline, then so is $W(\mu^{-1})$.
\end{Thm}
\begin{proof}
Let $L/K$ be a finite Galois extension such that $W$ is semi-stable as a $B^{\otimes E}_{|L}$-pair, so that \cite[5.1.7]{Fon94b} implies that we have an isomorphism of $E\dash(\varphi,N,\Gal(L/K))$-modules \[\Schur^u(D_{\st,L}(W))\riso D_{\st,L}(\Schur^u(W))\] 

If $\Schur^u(W)$ is semi-stable, then proposition \ref{sstreps} implies that $I_{L/K}$ acts trivially on $\Schur^u(D_{\st,L}(W))$.  Lemma \ref{lemabove} implies that $I_{L/K}$ acts on $D_{\st,L}(W)$ via a character $\eta:I_{L/K}\to E^{\times}$.  By the same reasoning as in the proof theorem \ref{ssttensors}, there is a finite extension $F/E$ and a character $\mu:\Gal(L/K)\to F^\times$ such that $\mu|_{I_{L/K}}=\eta$.  By proposition \ref{sstreps}, $W(\mu^{-1})$ is semi-stable. 

If $\Schur^u(W)$ is crystalline, then one also has that $N=0$ on $\Schur^u(D_{\st,L}(W))$.  Lemma \ref{lemabove} implies that $N=0$ on $D_{\st,L}(W)$, which implies the same for $D_{\st,L}(W(\mu^{-1}))$, so that $W(\mu^{-1})$ is crystalline.
\end{proof}

Theorem \ref{sstschur} implies the following.

\begin{Cor}
Let $V$ be a potentially semi-stable $E$-linear representation of $G_K$ such that $\dim_E V\geq r(u)$.  If the $E$-linear representation $\Schur^u(V)$ of $G_K$ is semi-stable, then there is a finite extension $F/E$ and a character $\mu:G_K\to F^\times$ such that the $F$-linear representation $V(\mu^{-1})$ of $G_K$ is semi-stable. If, moreover, $\Schur^u(V)$ is crystalline, then so is $V(\mu^{-1})$.
\end{Cor}

\renewcommand\refname{References}

\end{document}